\documentclass[11pt,letterpaper]{amsart}

\usepackage{dcolumn,indentfirst}
\usepackage[colorlinks=false,linkcolor=black,citecolor=black,urlcolor=black,bookmarksopen=true,linktocpage=true,pdfstartview=FitH]{hyperref}
\usepackage{amsmath,amssymb,amscd,amsthm,amsfonts,mathrsfs}
\usepackage{color,graphicx,xcolor,graphics,subfigure,extarrows,caption2}
\usepackage{titletoc}
\newtheorem{thm}{Theorem}[section]
\newtheorem*{main}{Main Theorem}
\newtheorem*{ky}{Lemma 5.1}
\newtheorem*{ky2}{Lemma 5.2}

\newtheorem*{conj}{Finite Global Attractor Conjecture}
\newtheorem{prop}[thm]{Proposition}
\newtheorem{lema}[thm]{Lemma}
\newtheorem*{lema6}{Thick-thin decomposition Theorem}

\theoremstyle{definition}
\newtheorem*{defi}{Definition}
\newtheorem{rmk}{Remark}

\usepackage{enumerate,enumitem}

\makeatletter\@addtoreset{equation}{section}\makeatother

\begin{document}

\author{Shuyi Wang \&  Gaofei Zhang}

\address{School of Mathematical Sciences, Ocean University of China, Qingdao 266100, China}
\email{sywangmath@aliyun.com}
\address{Department of Mathematics, QuFu Normal University, Qufu  273165,   P. R. China}
\email{zhanggf@hotmail.com}

\title[]{  Polynomial curve systems are exponentially decaying  }

\begin{abstract}
The existence of a finite global attractor for polynomial curve systems has been known since the work of Belk \emph{et al.}~\cite{BLMW}.
However, except in the hyperbolic case, the rate at which the pullback of a curve under a polynomial converges to the attractor has remained unclear.

In this paper, we introduce the notions of \emph{quick returns} and \emph{barrier lakes} to analyze the combinatorial models of curves.
These concepts allow us to show that if a certain number of successive pullbacks do not decrease the complexity of the curve by a definite proportion, then the curve admits a \textit{thick--thin decomposition}: most of the curve is organized into finitely many disjoint annuli whose core curves have bounded homotopy type.
These annuli are essentially the roundabouts for the polynomial case conjecturally predicted by Bartholdi--Dudko--Pilgrim in their strategy to attack the general finite global attractor conjecture~\cite{BDP}.
In this case, we can show that some number of successive pullbacks must decrease the complexity of the curve by a definite factor.

This implies that the complexity of a curve $C$ decreases exponentially under iteration of the pullback by a polynomial $f$:
\[
N_{\mathcal{F}}(\eta) \le A \, N_{\mathcal{F}}(C) \, e^{-n \delta} + D, \qquad \forall n\ge 1,
\]
where $\mathcal{F}$ is an admissible family of separation arcs,
$N_{\mathcal{F}}(\cdot)$ denotes the minimal intersection number of the curves in its homotopy class with the arcs in $\mathcal{F}$,
$\delta>0$ is a constant depending only on $f$,
$A, D > 0$ are constants depending only on $\mathcal{F}$ and $f$,
and $\eta$ is any component of $f^{-n}(C)$.

Consequently, the pullback of a curve contracts exponentially to the attractor. In particular, this provides a quantitative proof of the finite global attractor conjecture for the polynomial case.
\end{abstract}



\date{\today}



\maketitle


\section{Introduction}
The dynamics of the curve system for a Thurston map is closely related to the rationality of a topological map and the boundary behavior of the action of the Thurston pullback map on its Teichm\"uller space. Both aspects have been studied extensively in recent years; see \cite{BDP,BN,BFH,BLMW,BHI,BHL,CFPP,F,HS,KPS,L,P1,P0,P2,PT,RSY,S,Se,SY,ST,T}¡ªand this list is by no means complete. This connection leads to the following open problem in holomorphic dynamics \cite{BDP,BHI,P}.

\begin{conj} Let $f$ be a post-critically finite rational map and $P_f$ its post-critical set. Suppose $f$ is not a flexible Latt\'{e}s map. Then there exists a finite set of non-peripheral curves $\gamma_i$, $i = 1, \dots, n$, such that for any non-peripheral curve $\gamma$ in $\widehat{\Bbb C} \setminus P_f$  and all $k$   large enough, any non-peripheral component of $f^{-k}(\gamma)$ must be homotopic to some $\gamma_i$.
\end{conj}
Progress has been made toward resolving this conjecture.  Kelsey and Lodge verified it for quadratic non-Latt¨¨s maps with four post-critical points \cite{KL}. Hlushchanka proved the conjecture for all critically fixed rational maps \cite{H}. Pilgrim established that the conjecture holds when the associated virtual endomorphism on the mapping class group is contracting, as well as for quadratic polynomials with a periodic critical point, and for three specific quadratic polynomials \cite{P0}. In addition, Bonk and his collaborators made substantial contributions to the case of four post-critical points \cite{BHI}\cite{BHL}; this case has recently been completely solved by Bartholdi, Dudko and Pilgrim \cite{BDP}. There is also related work on marked rational maps with three post-critical points; see \cite{Sm}. Most recently, by adapting the method developed in this paper, the existence of a finite global attractor was verified for a class of rational maps obtained by gluing two polynomials along the boundaries of two finite immediate attracting basins \cite{W}.

The conjecture for general  polynomial case
was recently solved by  Belk $\emph{et al}$, by using tree lifting algorithm \cite{BLMW}.  Their proof is based on the iteration of a simplicial map on the complex of trees.  Let us summarize the idea of the proof  as follows. The reader may refer to \cite{BLMW} for the relative details.  Assume that  $C\subset \Bbb C - P_f$ is a
 non-peripheral curve. Then  $C$ is the  boundary of a Jordan   neighborhood of a subforest of the
 tree containing $P_f$ as vertices. As a consequence of the polynomial mapping property,   each
 pull back of  $C$  by $f$ is the boundary of a Jordan  neighborhood of some subforest of the lifting of the tree by $f$.
On the other hand,  the tree lifting operator  $\lambda_f: \mathcal{T} \to \mathcal{T}$ is a simplicial map where $\mathcal{T}$ is
 the complex of trees containing $P_f$ as vertices.    It was proved in
  \cite{BLMW} that, $\lambda_f$ does not increase  a metric on $\mathcal{T}$ and the iteration of $\lambda_f$ on any tree must eventually   converges to some small neighborhood of a Hubbard vertex, which contains finitely many trees  when $f$ is a polynomial.  So  the pull back of a curve  eventually becomes the boundary of a Jordan neighborhood of
   some subforest of one of the
   finitely many trees. The finite global attractor conjecture for the polynomial case thus follows.  However,   as pointed out in \cite{BLMW},   the proof
  does not tell how fast the iteration of $\lambda_f$ on a tree converges to a Hubbard vertex, and correspondingly,  it  is not clear   how fast a curve, under the iteration of the pull back,  converges to the attractor $\{\gamma_i\}$. In fact,  except the hyperbolic case,  for which Nekrashevych's work
   \cite{N} may be adapted to show that   $\lambda_f$ is an exponential contraction,
 it is not clear in general  if  $\lambda_f$  contracts the metric in $\mathcal{T}$ by a definite rate (see \emph{Overview of the proof} \cite{BLMW})

 \begin{quote}
 \itshape   It would be interesting to understand the global rate of the contraction of $\lambda_f$. Nekrashevych studies a cell complex that is closed related to our $\mathcal{T}_n$ and shows that a hyperbolic polynomial acts on it with exponential contraction. We suspect that a similar argument would show that a hyperbolic polynomial gives exponential contraction of $\mathcal{T}_n$; however, the non-hyperbolic case is more mysterious.
 \end{quote}
         This work is motivated by the above question. The goal of the paper is to show that, in both hyperbolic and non-hyperbolic cases, the complexity of a curve always decreases exponentially. Consequently, the iteration of the pullback of a curve converges exponentially to the attractor $\{\gamma_i\}$. In particular, this gives an quantitative   proof of the finite global attractor conjecture for the polynomial case.

Before stating the main theorem, let us first use the intersection number to define the \emph{complexity} of a curve with respect to a family of separating arcs. A similar notion was used in previous works (see, e.g., \cite{H} \cite{BHI}).

   Let $f$ be a post-critically finite polynomial and let $P_f$ be the post-critical set of $f$.  We call a simple   arc $L \subset \mathbb{C}$ \emph{separating} if both ends of $L$ approach infinity. We denote by $N(\cdot, \cdot)$ the intersection number of two curves(or two family of curves). By homotopy, we always assume that all intersections considered are transversal. In particular, for a separating arc $L$  and a simple closed curve
    $C \subset \mathbb{C} \setminus P_f$, $N(C, L)$ denotes the intersection number of $C$ and $L$. For a finite set $\mathcal{F}$ of separating arcs, set
\[
N_{\mathcal{F}}(C) = \min_{\gamma} \sum_{L \in \mathcal{F}} N(\gamma, L),
\]
where the minimum is taken over all $\gamma$ that are homotopic to $C$ in $\mathbb{C} \setminus P_f$. We call the quantity $N_{\mathcal{F}}(C)$ the \emph{complexity} of $C$ with respect to $\mathcal{F}$. To ensure that $N_{\mathcal{F}}(C)$ truly reflects the complexity of the curve $C$, we impose some conditions on $\mathcal{F}$.

\begin{defi}
Let $\mathcal{F}$ be a family of separating arcs. We call $\mathcal{F}$ \emph{admissible} if each component of $\mathbb{C} - \bigcup_{L \in \mathcal{F}} L$ contains at most one point of $P_f$, and moreover, any two separating arcs in $\mathcal{F}$ are either disjoint or intersect each other only at points of $P_f$.
\end{defi}

\begin{main}\label{B1}
Let $f$ be a post-critically finite polynomial of degree $d \ge 2$, let $P_f$ be its post-critical set, and let $\mathcal{F}$ be an admissible family of separating arcs. Then there exist constants $\delta = \delta(f) > 0$ depending only on $f$, and $A = A(f, \mathcal{F}) > 1$, $D = D(f, \mathcal{F}) > 0$ depending only on $f$ and $\mathcal{F}$, such that for any simple closed curve $C \subset \mathbb{C} - P_f$ and any integer $n \ge 1$,
\[
N_{\mathcal{F}}(\eta) \le A \, N_{\mathcal{F}}(C) \, e^{-n \delta} + D,
\]
where $\eta$ is any component of $f^{-n}(C)$.
\end{main}

  Let us explain the constants appearing in the main theorem. The constant $D$ is the upper bound of the complexity of the attractors $\{\gamma_i\}$ with respect to $\mathcal{F}$; it reflects the existence and finiteness of the global attractor of $f$. The constant $A$ is dispensable, since in general, not every single pullback, but rather a certain number of successive pullbacks, decreases the complexity of a curve at a definite rate.

Finally, we would like to present a question posed by Shishikura in an online seminar.

\noindent \textbf{Question.} (Shishikura, \cite{Sh}) What is the relation between $\delta$ and the core entropy of the Hubbard tree for $f$?

 \tableofcontents

  \section{The idea of of the proof}
The proof is based on a detailed analysis of the combinatorial model of the curves. We first choose a specific admissible arc family $\mathcal{F}$ so that each separating arc in $\mathcal{F}$ is eventually mapped to a separating arc $L$ which is either the union of two periodic external rays landing at the same point, or the union of two periodic internal rays emanating from a super-attracting periodic point together with the two external rays extending them respectively. The construction relies on the fact that the image of each edge of the Hubbard tree under forward iteration of the polynomial eventually covers either an expanding edge or an attracting edge (see Lemma~\ref{even-p}). Let $\mathcal{F}_0$ denote the family of such $L$. In general, $\mathcal{F}_0$ is no longer admissible; however, we will see in Section~4 that the complexity with respect to $\mathcal{F}$ is dominated by the complexity with respect to $\mathcal{F}_0$, and thus it suffices to prove the main theorem for $\mathcal{F}_0$.

By construction, the arcs in $\mathcal{F}_0$ are all periodic. Let $p$ be the common period of these arcs. Then $f^p$ maps each arc $L$ in $\mathcal{F}_0$ homeomorphically onto itself, preserving the orientation.  Let $C$ be a minimal representative in its homotopy class.
It is clear that $C$ minimizes $N(C, L)$ for all $L \in \mathcal{F}_0$. We may choose $L \in \mathcal{F}_0$ such that
\begin{equation}\label{bbaa}
N(C, L) = \max_{\Theta \in \mathcal{F}_0} N(C, \Theta).
\end{equation}
It suffices to prove that after a fixed number of pullbacks by $f$, the quantity $N(C, L)$ decreases by a definite factor. To see this, we take a separating arc $L'$ close to $L$ so that $f^{lp}$, for some integer $l \ge 1$, maps $L'$ homeomorphically onto $L$ preserving orientation, and moreover, the strip bounded by $L$ and $L'$ contains no post-critical points of $f$ (see Figure~1 for an illustration).  Let $F = f^{lp}$. Consequently,   $N_{\mathcal{F}_0}(C)$ is not increased under the pullback of $F$.   The core idea is to show that, once the complexity of $C$ is sufficiently large, then either the pullback of $C$ by $F$ decreases $N(C, L)$ by a definite factor, or $C$ admits a $\emph{thick-thin decomposition}$:  most of $C$ is organized into finitely many ``thick disjoint annuli'', each containing a large number of long spirals of the same homotopy type. Here ``most'' refers to the part that contributes the most to $N(C, L)$. The proof is then completed by showing that, in the latter case, a  fixed number of successive pullbacks by F  necessarily yields a definite proportion of inefficient spirals, thereby reducing $N(C, L)$  by a definite factor.  We now explain the idea in more detail.

       \begin{figure}[!htpb]
  \setlength{\unitlength}{1mm}
  \begin{center}

 \includegraphics[width=135mm]{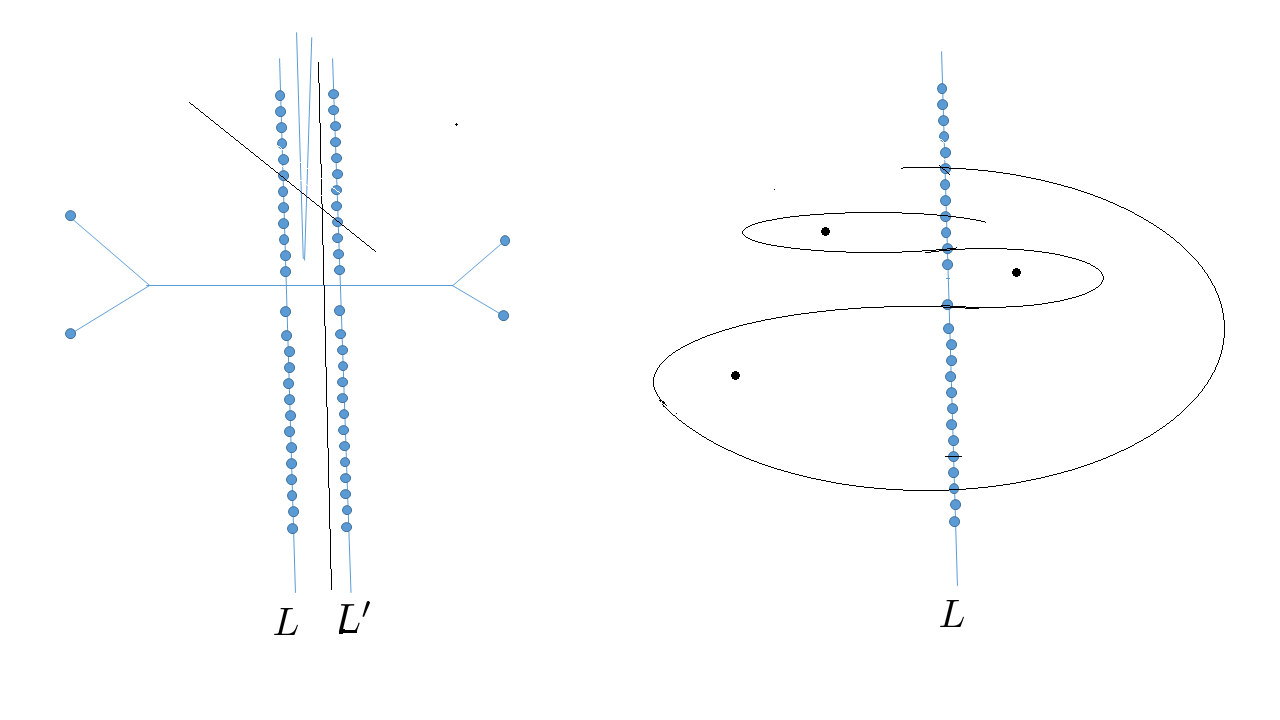}
  \caption{Quick returns -  the lines between $L$ and $L'$ are some pre-images of $L$. The dots on $L$ and $L'$ on the left  denote the preimages of the intersection points of $C$ and $L$ on the right. }
  \label{Figure-1}

  \end{center}
  \end{figure}

Let  $C^1$ be a component of $F^{-1}(C)$.
Since $F: L \to L$ is an orientation-preserving homeomorphism, we must have $N(C^1, L) \le N(C, L)$.
The starting point of this work is the following key observation.

\begin{ky}\label{pa}
Suppose $d \ge 2$ is the degree of $f$, and let $l, p$ be the integers determined as above.
Then whenever $N(C, L) > d^{lp}$, we must have
\[
\min_{\Gamma} N(\Gamma, L) < N(C, L),
\]
where the minimum is taken over all $\Gamma$ in the homotopy class of $C^1$.
\end{ky}

The proof proceeds by considering the pieces $S$ of $C^1$ that lie in the strip bounded by $L$ and $L'$, with the two endpoints belonging to $L$ and $L'$, respectively.
Then $F(S)$ is an arc piece of $C$ whose two endpoints lie on $L$.
If the lemma were false, then for every such piece the two endpoints of $F(S)$ would coincide.
This would imply that the degree of $F: C^1 \to C$ is at least $N(C, L)$, contradicting the fact that the degree of $F$ equals $d^{lp}$.
See Figure~1 for an illustration of this idea, and see Section~5 for a detailed discussion.

To understand how the complexity of the curve decreases, we explore the above idea further.
Let $K(S)$ denote the number of intersection points of $C$ and $L$ that lie between the two endpoints of $F(S)$.
By a direct argument, we obtain the following.

\begin{ky2}
For each such $S$, we have
\begin{equation} \label{bbcc}
\min_{\Gamma} N(\Gamma, L) + K(S) \le N(C, L),
\end{equation}
where the minimum is taken over all $\Gamma$ in the homotopy class of $C^1$.
\end{ky2}

Thus the main theorem follows if there exists a piece $S$ for which the quantity $K(S)/N(C, L)$ admits a positive lower bound depending only on $f$; see Lemma~\ref{QR}.

  \begin{figure}[htbp]
 \centering
 \begin{minipage}[t]{0.40\textwidth}
   \includegraphics[width=\textwidth]{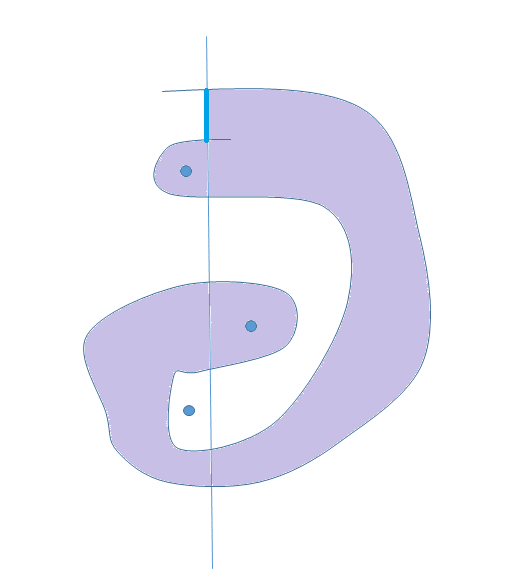}
   \caption{Type I barrier lake }
   \label{fig:image1}
 \end{minipage}
 \hfill
 \begin{minipage}[t]{0.40\textwidth}
   \includegraphics[width=\textwidth]{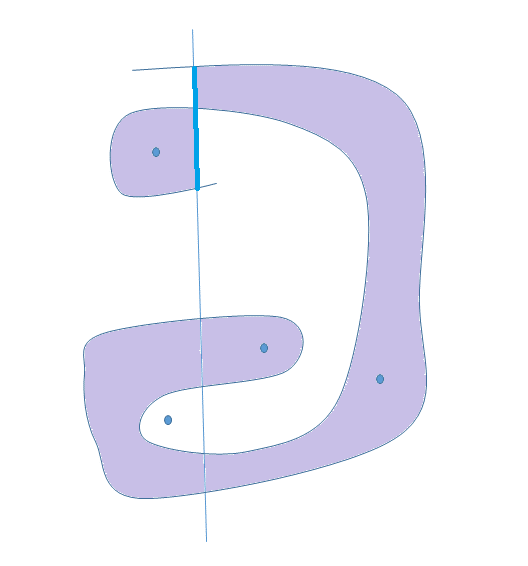}
   \caption{Type II barrier lake }
   \label{fig:image2}
 \end{minipage}
\end{figure}

  To obtain a contradiction, we fix constants $C_0, \epsilon > 0$ (to be specified in the final step) such that
\[
\max\{d^{lp}/C_0,\; \epsilon C_0^2 \} \ll 1/|P_f|,
\]
where $|P_f|$ denotes the number of points in $P_f$.

We say that $S$ (or $F(S)$) is a \emph{quick return} if $N(F(S), L) \le C_0$.
We may assume that
\begin{equation}\label{ccdd}
\min_{\Gamma} N(\Gamma, L) > (1 - \epsilon) N(C, L),
\end{equation}
where the minimum is taken over all $\Gamma$ in the homotopy class of $C^1$.
Indeed, if this were not the case, then from (\ref{bbaa}) the quantity $\mathcal{N}_{\mathcal{F}_0}(C^1)$ would be smaller than $\mathcal{N}_{\mathcal{F}_0}(C)$ by a definite factor determined by $\epsilon$.

We will see that for most $S$, $F(S)$ is a quick return (see Lemma~\ref{obs-2}).
From (\ref{bbcc}) and (\ref{ccdd}), for each such $S$ we have
\[
K(S) < \epsilon N(C, L).
\]
In this situation, we call the segment $I \subset L$ connecting the two endpoints of $F(S)$ an \emph{$\epsilon$-barrier}, and the union of the bounded components of $\mathbb{C} \setminus (F(S) \cup I)$ an \emph{$\epsilon$-barrier lake}.
We say that an $\epsilon$-barrier lake is of \emph{type I} if it does not separate post-critical points; otherwise, it is of \emph{type II}.
We will see that if $F(S)$ yields a type I barrier lake, then $F(S)$ is the union of finitely many adjacent loops (see Lemma~\ref{lp1}).

The essential idea of $\epsilon$-barrier lakes is as follows: once the curve passes through the barrier and enters the interior of the lake, the thinness of the barrier typically forces it to remain inside for a long time before exiting, unless the interior of the lake contains very few intersection points of $C$ and $L$; the same holds for the reverse direction. This allows us to organize most of the curve $C$ by a finite family of type II barrier lakes. These type II barrier lakes form a frame in which the type I barrier lakes are grouped into finitely many classes, each class consisting of a large number of type I barrier lakes whose loops have the same homotopy type. Here ``most'' means the part that contributes the most to $N(C, L)$. We will show that most of $C$ is the union of finitely many disjoint annuli, each of which consists of a number of long spirals with the same winding number and homotopy type.

\begin{lema6}\label{ann-1} Suppose $C$ is a simple closed curve in $\Bbb C - P_f$ and minimal with respect to $\mathcal{F}_0$, 
 and $C^1$ is a component of $F^{-1}(C)$ so that
$$
\min_{\Gamma \sim C^1} N(\Gamma, L) > (1 - \epsilon) N(C, L)
$$ where the $\min$ is taken in the homotopy class of $C^1$.
Then most part of $C$ can be organized into finitely many   disjoint annuli $A_i, 1 \le i \le  n ( \le |P_f|-3)$, each of which consisting of a number of long spirals with the same winding number and homotopy type, so that
$$
N(A_i, L) >  4 \epsilon C_0^2 N(C, L),  \:\: 1 \le i \le n,
$$
and
$$
 N(C, L) -  \sum_{1\le i \le n}   N(A_i, L)\le  (\epsilon +  d^{lp}/C_0 +8 |P_f|   C_0^2  \epsilon )N(C, L),
$$
and the  homotopy complexity of any loop in
  these $A_i$, with respect to $\mathcal{F}_0$,    has an upper bound  depending only on  $C_0$.
\end{lema6}

 We would like to connect the thick--thin decomposition presented here with that of Bartholdi--Dudko--Pilgrim~\cite{BDP}, as both are essentially the same phenomenon seen from different perspectives. From the analytic perspective of Bartholdi--Dudko--Pilgrim, when the pullback map fails to contract the norm of a tangent vector by a definite factor -- equivalently, when the pushforward operation does not cause a definite amount of loss of mass of the dual quadratic differential -- the quadratic differential must concentrate on finitely many disjoint fat annuli. Similar reasoning also appears in the iteration of the Thurston pullback for the Teichm\"{u}ller space of post-singularly finite exponentials \cite{HSS}. In our combinatorial setting, the analogous condition is that the intersection number with a separating arc decreases very slowly under pullback. This forces the curve, via the mechanism of quick returns and $\epsilon$-barrier lakes, to organize itself into a finite set of disjoint thick annuli, each consisting of many long spirals of identical homotopy type. Thus both viewpoints converge to a unifying principle: in the absence of effective contraction, the dynamical complexity must concentrate on a finite union of annular regions, where the dynamics appears in a repeated pattern.

We will utilize these ideas in Section 8 to finish the proof.
Suppose (\ref{ccdd}) holds. By the thick--thin decomposition theorem, $C$ admits a thick--thin decomposition. For $m \ge 1$, the pullback of $A_i$ by $F^m$ yields finitely many disjoint annuli $A_i^j$.
We consider only those $A_i^j$ whose loops are non-peripheral; in particular, with respect to the orbifold metric, their lengths have a positive lower bound depending only on $f$.
Since the complexity of the loops in $A_i$, with respect to $\mathcal{F}_0$, is bounded by some constant depending only on $C_0$, for $m \ge t$ (where $t$ is the integer in Lemma~\ref{w-2}), it follows from Lemma~\ref{w-2} that the complexity of the loops in $A_i^j$, with respect to $\mathcal{F}$, has an upper bound depending only on $C_0$. In particular, the length of the loops in $A_i^j$ with respect to the orbifold metric, up to homotopy, has an upper bound.

Since pullback by $F$ strictly decreases the orbifold metric in a neighborhood of the Julia set, the degree of $F^m: A_i^j \to A_i$, and thus the number of spirals in $A_i^j$, grows exponentially with $m$.
On the other hand, it will be shown that the number of efficient spirals in $A_i^j$ is bounded. Here ``efficient'' means that the spiral cannot be removed by deformation in $\mathbb{C} - P_f$.

Therefore, by choosing $m$ large enough, at least half of the spirals in all $A_i^j$ become inefficient.
The main theorem then follows.

\section{Hubbard tree}
In this his section we will give a background for the  polynomial dynamics, especially for post-critically finite polynomials.
Suppose $f$ is a degree $d$ polynomial for some $d \ge 2$. We call the set
$$
K(f) = \{z\: |\: f^n(z) \nrightarrow \infty \hbox{ as } n \to \infty\}
$$ the filled-in Julia set of $f$.  It is clear that $K(f)$ is a compact set. We say a point $z\in \Bbb C$ is a critical point of $f$ if $f'(z) = 0$.
Let $$\Omega_f = \{z\in \Bbb C\:|\: f'(z) = 0\}$$ denote the set of the critical points of $f$. We call
$$
P(f) = \bigcup_{n=1}^\infty f^n(\Omega(f))
$$ the post-critical set of $f$. It is known that
$K(f)$ is connected if and only if $P(f)$ is a bounded set.  We say $f$ is $\emph{post-critically finite}$, abbreviated as PCF,  if $P(f)$ is a finite set.
It follows that $K(f)$ is a connected set when $f$ is post-critically finite.

Suppose $K(f)$ is connected. Let $\Delta$ be the unit disk. Then there is a conformal isomorphism $\Phi: \Bbb C - K(f) \to \Bbb C - \overline{\Delta}$ so that the following diagram commutes.
$$
 \begin{CD}
        {\Bbb C} - K(f)         @  > \Phi     >  >    {\Bbb C} - \overline{\Delta}       \\
           @V  f  VV                         @VV z \mapsto z^d V\\
        {\Bbb C} - K(f)          @  >\Phi   >  >        {\Bbb C} - \overline{\Delta}
     \end{CD}
     $$
The family of the
radial rays $\{r e^{i \theta}\:|\: r > 1\}$ is invariant under the power map $z \mapsto z^d$. Thus the family of the rays
$$
\{R_\theta\} = \{\Phi^{-1} (r e^{i \theta})\:|\: r > 1\}
$$ is invariant under $f$. We call $\{R_\theta\}$  the external rays for $f$. So $f$ maps one external ray homeomorphically onto some other one.
Later we shall see that this property will be essentially used in our proof.

An external ray $R_\theta$  is called periodic of period $p$ if
$$
f^p(R_\theta) = R_\theta.
$$
By  Douady's theory (see Milnor's book, $\S18$), for each periodic repelling point, there is at least one periodic ray landing on it.

For PCF polynomials, the critical orbits are bounded and thus $K(f)$ is connected. Moreover, the dynamics of $f$ can be coded by finitely many combinatorial data of $f$ which are organized by a finite graph - the Hubbard tree for $f$.

Let us give a short description of the concept of Hubbard tree. The reader may refer to \cite{Po} for more details. Suppose $f$ is a PCF polynomial. Then there is a smallest finite tree which is forward invariant - the set of vertices and the set of edges are both finite, such that the tree is embedding into the filled-in Julia set and moreover, the set of vertices contains all the points in the post-critical set.

\subsection{Expanding and Attracting Edges for  Hubbard Trees}

  \begin{figure}[htbp]
 \centering
 \begin{minipage}[t]{0.40\textwidth}
   \includegraphics[width=\textwidth]{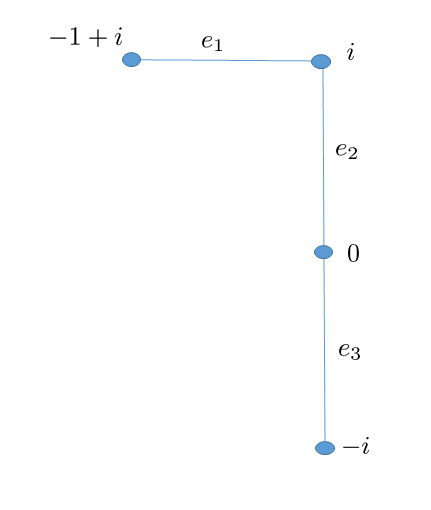}
   \caption{The topological Hubbard tree for $z^2 + i$}
   \label{fig:image1}
 \end{minipage}
 \hfill
 \begin{minipage}[t]{0.40\textwidth}
   \includegraphics[width=\textwidth]{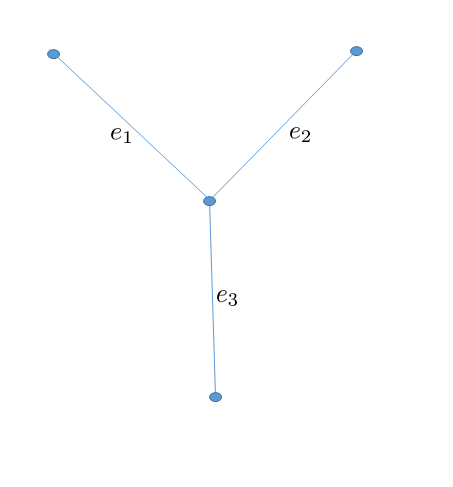}
   \caption{The topological Hubbard tree for the Douady-rabbit Julia set}
   \label{fig:image2}
 \end{minipage}
\end{figure}

The quadratic polynomial $f(z)=z^2+i$ is postcritically finite, with the critical orbit
\[
0 \to i \to -1+i \to -i \to -1+i.
\]
In most cases, we content ourselves with the topological Hubbard tree, where the edges of the Hubbard tree are replaced by straight line segments in the plane. The topological Hubbard tree for the above polynomial is illustrated in Figure~4.

For this example, a simple computation yields the transformation of the edges:
\[
e_1 \to e_1 + e_2 + e_3,\quad e_2 \to e_1,\quad e_3 \to e_1.
\]
Thus each edge eventually covers itself more than once after a number of iterations. That is, for each $e_i$, there are two disjoint subintervals of $e_i$, say $I$ and $J$, and an integer $n \ge 1$ such that $f^n: I \to e_i$ and $f^n: J \to e_i$ are homeomorphisms.  We call such an edge an \emph{expanding edge}. By definition, we clearly have the following proposition.

\begin{prop}\label{expd}
Suppose $e$ is an expanding edge. Then $e$ contains a repelling periodic point $z$ in the interior of $e$ with period $p$ so that   the following property holds: for any small interval neighborhood $I$ of $z$, there exist two disjoint subintervals $I_1 , I_2 \subset I$ with $z \in I_1$ and an integer $l \ge 1$ such that $f^{lp} : I_1 \to e$ and $f^{lp}: I_2\to e$ are homeomorphisms.
\end{prop}

There exist edges that are eventually mapped homeomorphically onto themselves. Since a PCF polynomial expands the orbifold metric, such an edge must have at least one of its endpoints being a super-attracting periodic point. We call such an edge an \emph{attracting edge}.

In the case of Douady's rabbit, we see the transformation of the edges:
\[
e_1 \to e_2,\quad e_2 \to e_3,\quad e_3 \to e_1.
\]
Hence each edge is an attracting edge (see Figure~5). There are Hubbard trees for which some edges are neither expanding nor attracting. Nevertheless, we have the following lemma.

\begin{lema}\label{even-p}
Suppose $e$ is an edge of a Hubbard tree. Then  there is a sub-interval $I$ of $e$ and an $n \ge 1$ such that $f^n$ maps $I$ homeomorphically onto an expanding edge or an attracting edge.
\end{lema}

\begin{proof}
Let $H$ be the Hubbard tree and $e$ be an edge of $H$. For each $k \ge 1$, the set $H_k = f^k(e)$ is connected and thus a subtree of $H$. Since $H$ contains only finitely many distinct subtrees, there exists a subsequence $\{k'\}$ such that all $H_{k'}$ are the same subtree $H'$. If $H'$ consists of a single edge, then that edge must be either expanding or attracting.

Otherwise, repeat the same process for each edge of $H'$. Then either the image of every edge of $H'$ under iteration of $f$ eventually covers the whole $H'$, or there exists an edge $e'$ of $H'$ and a subsequence $\{k''\}$ such that $f^{k''}(e') = H''$ is a proper subset of $H'$. In the first case, every edge in $H'$ is an expanding edge. In the second case, repeat the same procedure for $H''$. Since $H$ is a finite set, we must eventually reach either a subtree whose edges are all expanding, or a subtree consisting of exactly one edge, which is either attracting or expanding.
\end{proof}

\section{A specific $\mathcal{F}$}

The next lemma implies that we need only prove the main theorem for
some specific $\mathcal{F}$.

\begin{lema}\label{df} Suppose $\mathcal{F}$ and $\mathcal{F}'$ are two admissible family of separating arcs. Then for all simple closed curve $C \in \Bbb C - P_f$,
\[
N_{\mathcal{F}}(C) \asymp N_{\mathcal{F}'}(C),
\]
with the constant depending only on $\mathcal{F}$ and $\mathcal{F}'$.
\end{lema}

\begin{proof}
First note that deforming the separating arcs in $\mathcal{F}$ and $\mathcal{F}'$ by homotopy in $\mathbb{C} \setminus P_f$ does not affect $N_{\mathcal{F}}(\cdot)$ and $N_{\mathcal{F}'}(\cdot)$. Thus we may assume all arcs in $\mathcal{F}$ and $\mathcal{F}'$ are smooth.
We may further assume that $C$ minimizes the complexity with respect to $\mathcal{F}'$.
Let $U_i$, $1 \le i \le m$, denote the components of $\mathbb{C} \setminus \mathcal{F}'$.
Let $\{\Gamma_j\}$ be all the arc components of $C \cap U_i$ for $1 \le i \le m$.
Then $N_{\mathcal{F}'}(C)$ equals the number of elements in $\{\Gamma_j\}$.

Now, by deforming the separating arcs in $\mathcal{F}$ by homotopy in $\mathbb{C} \setminus P_f$ if necessary, we may assume that for each $L \in \mathcal{F}$, the number of arc components of $L \cap U_i$ is finite, say $k_L^i$.
Set
\[
K = \max_{L \in \mathcal{F},\; 1 \le i \le m} k_L^i.
\]
For each $\Gamma_j$, by deforming $\Gamma_j$ in $\mathbb{C} \setminus P_f$ with its two endpoints fixed, the intersection number of $\overline{\Gamma_j}$ with $L$ is at most $2K$.
Here $\overline{\Gamma_j}$ is the union of $\Gamma_j$ and its endpoints.
This holds because $U_i$ contains at most one postcritical point; up to homotopy in $\mathbb{C} \setminus P_f$ fixing the endpoints, the intersection number of $\overline{\Gamma_j}$ with the closure of each arc component of $L \cap U_i$ is at most two.
Consequently,
\[
N(L,C) \le 2K \, N_{\mathcal{F}'}(C),
\]
and therefore
\[
N_{\mathcal{F}}(C) \le  \sum_{L\in\mathcal{F}} N(L,C) \le 2\,|\mathcal{F}|\,K\, N_{\mathcal{F}'}(C),
\]
where $|\mathcal{F}|$ denotes the number of separating arcs in $\mathcal{F}$.
The opposite inequality is proved in the same way.
\end{proof}

Before we specify an admissible family $\mathcal{F}$, we need some notions regarding the Hubbard tree.

Suppose $H$ is the Hubbard tree for $f$.
By Lemma~\ref{even-p}, for every edge $e$ of $H$, there exists $t \ge 1$ such that $f^t(e)$ covers either an expanding edge or an attracting edge.

There are two cases.

In the first case,
$f^t(e)$ covers an expanding edge. In this case, let $z$ be the repelling periodic point in the interior of the expanding edge whose existence is guaranteed by Proposition~\ref{expd}. Then there are two periodic rays landing on $z$, say $R_1$ and $R_2$, such that $R_1 \cup \{z\} \cup R_2$ separates the expanding edge. Let $x$ be an interior point of $e$ with $f^t(x)=z$. Pulling back $R_1$ and $R_2$ by $f^t$ yields two external rays landing on $x$, denoted $\tilde{R}_1$ and $\tilde{R}_2$. Now define the separating arc $L_e$ associated to $e$ as $\tilde{R}_1 \cup \{x\} \cup \tilde{R}_2$. Clearly, $L_e$ separates $e$.

In the second case,
$f^t(e)$ covers an attracting edge. Then there exists a point $x \in e$ that is mapped into some super-attracting cycle by $f^t$. We then define $L_e$ as the union of two (pre-)internal rays starting from $x$ and two external rays extending them respectively, so that $f^t(L_e)$ is the union of two periodic internal rays starting from $f^t(x)$ and two periodic external rays extending those internal rays, and moreover, $f^t(L_e)$ does not contain any other post-critical point except $f^t(x)$.

Now, by taking sufficiently many such separating arcs, (and in particular, we may need to take more than one separating arc for each edge in the second case), we obtain that
\[
\mathcal{F} = \{L_e\}
\]
is an admissible family.

Let $C \subset \mathbb{C} \setminus P_f$ be a simple closed curve and let $C'$ be a component of $f^{-1}(C)$.
Note that for each $L \in \mathcal{F}$, the image $f(L)$ is the union of two external rays and thus a separating arc.
The same argument as in the proof of Lemma~\ref{df} shows that
\[
N(C, f(L)) \preceq N_{\mathcal{F}}(C),
\]
with the constant depending only on $\mathcal{F}$.
Since $f$ maps $L$ homeomorphically onto $f(L)$, it follows that
\[
N(C', L) \le N(C, f(L)).
\]
Combining these estimates yields the following lemma.

\begin{lema}\label{w-1}
Suppose $C \subset \mathbb{C} \setminus P_f$ is a simple closed curve and $C'$ is a component of $f^{-1}(C)$. Then
\[
N_{\mathcal{F}}(C') \preceq N_{\mathcal{F}}(C),
\]
where the implied constant depends only on $\mathcal{F}$.
\end{lema}

By the construction of $\mathcal{F}$, for each $L \in \mathcal{F}$ there exists an integer $t_L \ge 1$ such that $f^{t_L}(L)$ is periodic.
Let $\mathcal{O}(f^{t_L}(L))$ denote the periodic cycle containing $f^{t_L}(L)$ and set
\[
\mathcal{F}_0 = \bigcup_{L \in \mathcal{F}} \mathcal{O}(f^{t_L}(L)).
\]
In general, the family $\mathcal{F}_0$ does not satisfy the separation property required in the definition of an admissible family.
Nevertheless, each $L \in \mathcal{F}_0$ is periodic and is associated to some expanding or attracting edge.
Let
\begin{equation}\label{tc}
t = \max\{t_L\}.
\end{equation}
Then $f^t$ maps each $L \in \mathcal{F}$ homeomorphically onto some element of $\mathcal{F}_0$.
Since all elements of $\mathcal{F}_0$ are periodic and $f$ maps an element of $\mathcal{F}_0$ homeomorphically onto another element of $\mathcal{F}_0$, it follows that the complexity with respect to $\mathcal{F}_0$ does not increase under pullback.

\begin{lema}\label{add-m}
For any simple closed curve $C \subset \mathbb{C} \setminus P_f$ and any component $C'$ of $f^{-1}(C)$, we have
\[
N_{\mathcal{F}_0}(C') \le N_{\mathcal{F}_0}(C).
\]
\end{lema}

\begin{lema}\label{w-2}
Suppose $C \subset \mathbb{C} \setminus P_f$ is a non-peripheral curve and $C'$ is a component of $f^{-t}(C)$ with $t = \max\{t_L\}$. Then
\[
N_{\mathcal{F}}(C') \preceq N_{\mathcal{F}_0}(C) \preceq N_{\mathcal{F}}(C),
\]
where the implied constant depends only on $\mathcal{F}$ and $\mathcal{F}_0$.
\end{lema}

\begin{proof}
The first inequality holds because $f^t$ maps each $L \in \mathcal{F}$ homeomorphically onto $f^t(L) \in \mathcal{F}_0$, and because $N(C', L) \le N(C, f^t(L)) \le N_{\mathcal{F}_0}(C)$. The second inequality follows from Lemma~\ref{df}, since one can add more separating arcs to $\mathcal{F}_0$ to make it an admissible family.
\end{proof}

Note that $\mathcal{F}$ is a specific family, so all constants depending on $\mathcal{F}$ and hence on $\mathcal{F}_0$ ultimately depend only on $f$. The main theorem follows once we prove the following lemma.

\begin{lema}\label{final-re}
There exist $q \ge 1$, $\delta > 0$, $A > 0$, and $D > 0$, all depending only on $f$, such that for any simple closed curve $C \subset \mathbb{C} \setminus P_f$ and any $k \ge 1$,
\[
N_{\mathcal{F}_0}(\eta) \le A N_{\mathcal{F}_0}(C) e^{-k\delta} + D,
\]
where $\eta$ is any component of $f^{-kq}(C)$.
\end{lema}

In fact, by Lemma~\ref{df}, it suffices to prove the main theorem for the specific family $\mathcal{F}$. Now suppose $\eta$ is a component of $f^{-(kq + t)}(C)$ with $t$ be the integer in (\ref{tc}) and  $k \ge 1$. Then $f^t(\eta)$ is a component of $f^{-kq}(C)$. Assuming Lemma~\ref{final-re} and using the first inequality of Lemma~\ref{w-2}, we obtain
\[
N_{\mathcal{F}}(\eta) \preceq N_{\mathcal{F}_0}(f^t(\eta)) \le A N_{\mathcal{F}_0}(C) e^{-k\delta} + D = A N_{\mathcal{F}_0}(C) e^{-(\delta/q)(kq)} + D.
\]
By Lemma~\ref{w-1} and by multiplying $A$ and $D$ by a number depending only on $t$, $q$, and $\delta$ (all of which depend only on $f$), still denoted by $A$ and $D$ respectively, we get
\[
N_{\mathcal{F}}(\eta) \le A N_{\mathcal{F}_0}(C) e^{-(\delta/q)(qk + t + i)} + D
\]
for any component $\eta$ of $f^{-(qk + t + i)}(C)$ with $k \ge 1$ and $0 \le i \le q-1$. That is, there exist $A, D \ge 1$ depending only on $f$ such that for all $n \ge t+q$ and any component $\eta$ of $f^{-n}(C)$,
\[
N_{\mathcal{F}}(\eta) \le A N_{\mathcal{F}_0}(C) e^{-n \delta'} + D,
\]
where $\delta' = \delta/q$.

By the second inequality of Lemma~\ref{w-2} and by multiplying $A$ by a number depending only on $f$ (still denoted $A$), the above yields
\[
N_{\mathcal{F}}(\eta) \le A N_{\mathcal{F}}(C) e^{-n \delta'} + D
\]
for all $n \ge   t+q$. By Lemma~\ref{w-1} and by multiplying $A$ and $D$ by some factor depending only on $f$, we can ensure that the inequality also holds for all $1 \le n < t+q$. This proves the main theorem.

Lemma~\ref{final-re} can be further reduced to the following.

\begin{lema}\label{Bridge-Prop}
There exist $D > 0$, $0 < \delta < 1$, and an integer $m \ge 1$, all depending only on $f$, such that whenever $N_{\mathcal{F}_0}(C) > D$, we must have
\[
N_{\mathcal{F}_0}(C^m) < \delta N_{\mathcal{F}_0}(C),
\]
where $C^m$ is any component of $f^{-m}(C)$.
\end{lema}

Let  $C^0 = C$ and $\{C^n\}_{n\ge 1}$ be a sequence of curves so that $C^{k+1}$ is the pull back of $C^k$ by $f$. Now assume  Lemma~\ref{Bridge-Prop} and let $l \ge 0$ be the largest integer such that $N_{\mathcal{F}_0}(C^{lm}) > D$. Then for $1 \le k \le l+1$,
\[
N_{\mathcal{F}_0}(C^{km}) < \delta^k N_{\mathcal{F}_0}(C).
\]
For $n = km + i$ with $0 \le k \le l$ and $0 \le i \le m-1$, we have $k = (n-i)/m > (n-m)/m = -1 + n/m$, and thus $\delta^k < \delta^{-1} \delta^{n/m}$. Hence
\[
N_{\mathcal{F}_0}(C^{n}) \le N_{\mathcal{F}_0}(C^{km}) < \delta^{-1} \delta^{n/m} N_{\mathcal{F}_0}(C).
\]
For all $n \ge (l+1)m$, we have $N_{\mathcal{F}_0}(C^{n}) \le D$ by Lemma~\ref{add-m}. Lemma~\ref{final-re} follows.

\section{Quick  returns}

We have already seen in Lemma~\ref{add-m} that the complexity with respect to $\mathcal{F}_0$ is not increased by pullback.
The starting point of this work is based on an observation: when $N_{\mathcal{F}_0}(C)$ is large, it must become smaller after a definite number of pullbacks; see Lemma~\ref{pa}.

To see this, take $L \in \mathcal{F}_0$. By the construction of $\mathcal{F}_0$ in \S4, there are two cases.
In the first case, $L$ is the union of two rays landing on some repelling periodic point in the interior of an expanding edge $e$ of the Hubbard tree and separates $e$.
In the second case, $L$ is the union of two periodic internal rays starting from some super-attracting periodic point and two periodic external rays extending them, respectively.

Suppose we are in the first case, i.e., $L = R_1 \cup \{z\} \cup R_2$, where $R_1$ and $R_2$ are two periodic external rays of period $p$ landing at a repelling periodic point $z$ on an expanding edge $e$, and $L$ separates $e$.
Let $R_1'$ and $R_2'$ be two rays which land at some point $z' \in e$ near $z$ and are mapped to $R_1$ and $R_2$ by $f^{lp}$, respectively, where $l \ge 1$ is an integer, and moreover $f^{lp}$ preserves the orientation at $z'$.
This means that $z'$ and $z$ lie in the interior of the same edge $e$ and $f^{lp}$ maps a small edge neighbourhood of $z'$ to a small edge neighbourhood of $z$ preserving the orientation.
In fact, by the choice of $z$ and Proposition~\ref{expd}, there exist two disjoint edge intervals $I$ and $S$ in $e$ such that $z \in I$ and both $I$ and $S$ are mapped homeomorphically onto $e$ by $f^{lp}$ for some $l \ge 1$.
If $f^{lp}: S \to e$ preserves the orientation, we may take a point $z' \in S$ such that $f^{lp}(z') = z$.
Otherwise, there is a subinterval of $S$, say $T$, which is mapped homeomorphically onto $S$ by $f^{lp}$. Then $f^{2lp}$ maps $T$ homeomorphically onto $e$ with preserved orientation.
In this case we take a point $z' \in T$ such that $f^{2lp}(z') = z$.
Thus $R_1$ and $R_1'$ lie on one side of the edge $e$, and $R_2$ and $R_2'$ lie on the other side.
Set $L' = R_1' \cup \{z'\} \cup R_2'$. It is clear that the strip bounded by $L$ and $L'$ contains no point of $P_f$. See the left-hand side of Figure~1 for an illustration of $L$ and $L'$.

Now suppose we are in the second case. Then $L$ is the union of two periodic internal rays starting from some super-attracting periodic point $w$ and two periodic external rays extending them.
Let $a$ and $b$ be the two intersection points of $L$ and $\partial D$, where $D$ is the immediate basin of $w$, which is a Jordan domain.
Suppose $p$ is the period of $L$, which must be an integer multiple of the period of $w$.
Since $f^p: \partial D \to \partial D$ is conjugate to $z \mapsto z^k$ on the unit circle, with $k \ge 2$ an integer, we can choose $a', b' \in \partial D$ close to $a$ and $b$, respectively, such that $f^{lp}(a') = a$ and $f^{lp}(b') = b$ for some large $l$, and moreover $a'$ and $b'$ lie on the same side of $L$.
We then define $L'$ as the union of the two internal rays starting from $w$ and landing at $a'$ and $b'$, together with the two external rays extending them, so that $f^{lp}(L') = L$.
Clearly, in this case we can also assume that the two sectors bounded by $L$ and $L'$ contain no points of $P_f$. See Figure~6 for an illustration of $L$ and $L'$.

Now, for both cases, let $F = f^{lp}$. It follows that $F(L) = F(L') = L$. To fix the idea, we illustrate the quick returns below  by assuming we are in the first case; the idea is easily adapted to the second case.

   \begin{figure}[!htpb]
  \setlength{\unitlength}{1mm}
  \begin{center}

 \includegraphics[width=70mm]{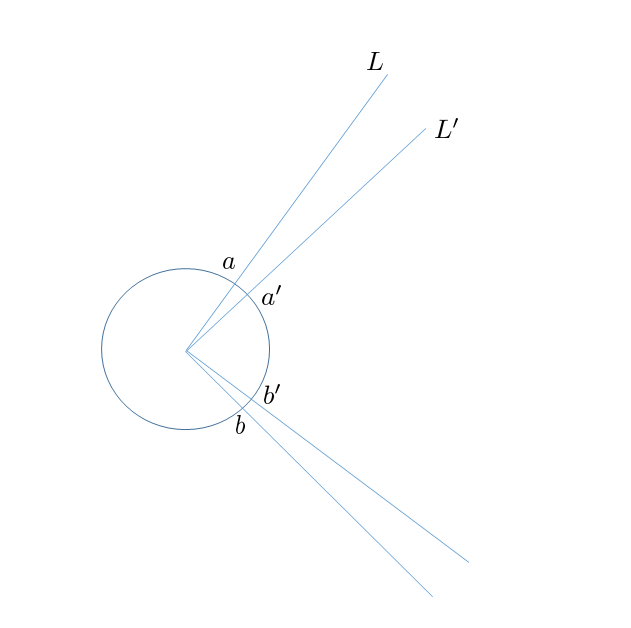}
  \caption{The $L$ and $L'$ for the second case }
  \label{Figure-1}

  \end{center}
  \end{figure}

Let $C \subset \mathbb{C} \setminus P_f$ be a simple closed curve and let $C^1$ be a component of $F^{-1}(C)$. Let $d \ge 2$ be the degree of $f$. Since $F: L \to L$ is a homeomorphism, we have
$$
N(C^1, L) \le N(C, L).
$$

\begin{lema}\label{pa}
Suppose $d \ge 2$ is the degree of $f$ and let $l, p$ be the integers determined as above. Then whenever $N(C, L) > d^{lp}$, we must have
$$
\min_{\Gamma} N(\Gamma, L) < N(C, L),
$$
where the minimum is taken over all $\Gamma$ in the homotopy class of $C^1$.
\end{lema}

Note that this implies that once the complexity is large enough, it must become strictly smaller after a fixed number of pullbacks. This gives another solution to the finite global attractor problem for the polynomial case.

\begin{proof}
We prove this by contradiction. Assume that equality holds, i.e., $N(C^1, L) = N(C, L)$. First note that each time $C^1$ enters the strip by crossing $L$ at some point $a$, it must then leave the strip by crossing $L'$ at some point $b$. Otherwise, since the strip bounded by $L$ and $L'$ contains no points of $P_f$, we could deform $C^1$ within its homotopy class to make $N(C^1, L)$ strictly smaller, yielding a contradiction. Thus the correspondence $a \mapsto b$ gives an injective map from the set of intersection points of $C^1$ with $L$ to the set of intersection points of $C^1$ with $L'$. Consequently,
$$
N(C^1, L') \ge N(C^1, L).
$$

On the other hand, since $F: L' \to L$ is a homeomorphism, we have $N(C^1, L') \le N(C, L)$. Because $N(C^1, L) = N(C, L)$, it follows that
$$
N(C^1, L') = N(C^1, L) = N(C, L).
$$

Since $F: L' \to L$ preserves orientation, for each pair of adjacent intersection points of $C^1$ with $L$ and $L'$, say $a$ and $b$, we must have $F(a) = F(b)$. It follows that the segment of $C^1$ lying between $a$ and $b$ and bounded inside the strip is mapped by $F$ onto the whole curve $C$. There are exactly $N(C^1, L)$ such disjoint `horizontal' segments in the strip, so the degree of $F: C^1 \to C$ is at least $N(C^1, L) = N(C, L)$, which by assumption is greater than $d^{lp}$. This contradicts the fact that the degree of $F$ is $d^{lp}$ (see the left side of Figure~1 for an illustration).
\end{proof}

The above observation is the starting point of this work.  Before we   explore the idea further, let us  introduce some terminology first.

We say that an intersection point of a simple closed curve in $\mathbb{C} \setminus P_f$ with $L$ is \emph{efficient} if it can not be removed by deforming the curve in $\mathbb{C} \setminus P_f$ via homotopy. We say that a simple closed curve $C \subset \mathbb{C} \setminus P_f$ is \emph{minimal} with respect to $\mathcal{F}_0$ if
$$
N_{\mathcal{F}_0}(C) = \sum_{L \in \mathcal{F}_0} N(C, L).
$$

 From now on we assume that $C$ is minimal with respect to $\mathcal{F}_0$.  By the minimality,  all intersection points of $C$ with every $L \in \mathcal{F}_0$ are efficient. Now take $L \in \mathcal{F}_0$ such that
\begin{equation}\label{mx}
N(C, L) = \max_{\Theta \in \mathcal{F}_0} N(C, \Theta).
\end{equation}
So, to prove Lemma~\ref{Bridge-Prop}, it suffices to show that after a certain number of pullbacks, $N(C, L)$ decreases by a definite factor, provided that it is large enough.

To see this, consider a segment $S$ in $C^1$ that lies in the strip as described in Figure~1. Suppose the two endpoints of $S$ are not mapped to the same point on $L$. Then as we move along $S$ from $L$ to $L'$, the image $F(S)$ starts at a point on $L$ and later returns to another point on $L$. Since $F: L' \to L$ preserves orientation, the two crossing directions at the starting and ending points must be the same. Let $K(S)$ denote the number of intersection points of $C$ with $L$ that lie between the starting point and the ending point.

\begin{lema}\label{obs-1}
$$
\min_{\Gamma} N(\Gamma, L) + K(S) \le N(C, L),
$$
where the minimum is taken over the homotopy class of $C^1$.
\end{lema}

\begin{proof}
Indeed, as we have seen above, each time $C^1$ makes an efficient crossing with $L$, it must then cross $L'$ immediately after the crossing with $L$. Now suppose $S$ is the curve segment illustrated in Figure~1. Below the segment $S$, the number of points on $L$ exceeds the number of points on $L'$ by $K(S)$. Hence there are at least $K(S)$ points on $L$ that are not efficient intersection points of $C^1$ with $L$. The lemma follows.
\end{proof}

Let $|P_f|$ denote the number of points in $P_f$. Now choose constants $\epsilon, C_0 > 0$, to be specified later, such that
\begin{equation}\label{CAS}
\max\{d^{lp}/C_0 , \epsilon C_0^2 \} \ll 1/|P_f|.
\end{equation}
We may assume that
\begin{equation}\label{BA}
\min_{\Gamma \sim C^1} N(\Gamma, L) > (1 - \epsilon) N(C, L),
\end{equation}
where the minimum is taken over the homotopy class of $C^1$.

From \ref{BA} it follows that $C^1$ has at least $(1 - \epsilon) N(C, L)$ arcs $S$ in the strip bounded by $L$ and $L'$ whose endpoints lie on $L$ and $L'$ respectively. Among these $S$, we are only concerned with those for which
$$
N(F(S), L) \le C_0.
$$

\begin{lema}\label{obs-2}
There are at least $(1 - \epsilon - d^{lp}/C_0 )N(C, L)$ segments $S$ such that
$$
N(F(S), L) \le C_0.
$$
\end{lema}

\begin{proof}
We prove this by contradiction. Suppose the statement were false. Then there would be $d^{lp} \cdot N(C, L) / C_0$ segments $S$ for which the inequality $N(F(S), L) \le C_0$ does not hold. But then the total number of intersection points of $F(C^1)$ with $L$, counted with multiplicity, would be greater than
$$
C_0 \cdot \frac{d^{lp} \cdot N(C, L)}{C_0} = d^{lp} \cdot N(C, L).
$$
This is impossible because the covering degree of $F$ is $d^{lp}$. The lemma follows.
\end{proof}

We call $F(S)$ a \emph{quick return} if $N(F(S), L) \le C_0$. By Lemma~\ref{obs-1}, if some quick return creates a definite gap between its starting point and its endpoint, then we obtain a definite reduction in the number of intersection points. So by Lemma~\ref{Bridge-Prop},  we have:

\begin{lema}\label{QR}
The main theorem follows if there exists some $\delta > 0$ such that, whenever $N(C, L)$ is sufficiently large, there is a segment $S$ satisfying $K(S) / N(C, L) > \delta$.
\end{lema}

 \section{$\epsilon$-Barrier lakes}

Recall that $C$ is minimal
 and $C^1$ is a component of $F^{-1}(C)$ so that  (\ref{BA})  holds.
By Lemma~\ref{QR}, the Main Theorem will follow if some quick return yields a gap compatible with $N(C, L)$. Hence we may assume that all quick returns produce thin gaps. Then we obtain  $\emph{barrier lakes}$, where the $\emph{barriers}$ are the parts of $L$ whose endpoints are the starting and ending points of the quick return.    In what follows, we assume  that all barriers are thin, each containing at most $\epsilon N(C,L)$ intersection points of $C$ and $L$. Let us explain this more precisely.

\begin{rmk}\label{npt}
The barrier does not contain post-critical point. This is true even if $L$   contains a post-critical point $w$, see Figure 6. In this case, the segment $S$ is contained in one of the two sectors bounded by $L$ and $L'$ and thus the two end points of $F(S)$  must belong to   one of the two components of $L - \{w\}$.
\end{rmk}
We say that a barrier lake $B$ \emph{separates} two post-critical points $x$ and $y$ if there exist two sublakes of $B$ containing $x$ and $y$, respectively. A barrier lake is called \emph{type I} if it does not separate any pair of post-critical points; otherwise it is called \emph{type II}. Since a type II barrier lake separates post-critical points by definition, it is clear that:

\begin{lema}\label{obs-3}
A type II barrier lake contains at least two post-critical points.
\end{lema}

Let $T \subset \mathbb{C}$ be a piece of curve segment such that both endpoints of $T$ lie on $L$ and have the same cross directions (Note this does not depend on the orientation of $T$). We can make $T$ into a closed curve by adjoining to $T$ the segment of $L$ that connects the two endpoints of $T$. We say that $T$ is a \textit{loop} if this closed curve is a simple closed curve. We say that two loops $T$ and $T'$ are homotopic to each other if the two simple closed curves associated to them are homotopic in $\mathbb{C} \setminus P_f$. Two loops are called \textit{adjacent} if the endpoint of one loop is the starting point of the other. A \textit{long spiral} $S$ is the union of many adjacent loops with the same homotopy type. We define the \textit{winding number} of a long spiral to be the number of loops in it.

The following lemma seems obvious. However, it is the foundation for the thick-thin decomposition in the next section; for completeness, we present a detailed proof.

\begin{lema}\label{lp1}
Suppose $F(S)$ produces a type~I barrier lake $B$. Then $F(S)$ is the union of finitely many adjacent loops, all of the same homotopy type and having the same intersection number with $L$. Moreover, the associated Jordan curve for each loop in $B$ is non-peripheral in $\overline{\mathbb{C}} \setminus P_f$; that is, $B$ contains at least two post-critical points, and the complement of $B$ contains at least one finite post-critical point.
\end{lema}

\begin{proof}
In the case where $F(S)$ does not cross the interior of the barrier, the barrier lake is a Jordan domain bounded by $F(S)$ and the barrier $I$. Then $F(S)$ consists of a single loop. Now assume that $F(S)$ crosses the interior of the barrier. We claim that for any two adjacent intersection points of $F(S)$ and $I$, taken in the order they appear along $F(S)$, the two cross directions are the same. We prove the claim by contradiction.

 Suppose that $A$, $B$, and $C$ are three adjacent intersection points of $F(S)$ and $I$, and that the two cross directions at $A$ and $B$ are different. Let $\Gamma$ be the arc of $F(S)$ connecting $A$ and $B$, and let $I_1$ be the segment of $I$ connecting $A$ and $B$. Similarly, let $\Theta$ be the arc of $F(S)$ connecting $B$ and $C$, and let $I_2$ be the segment of $I$ connecting $B$ and $C$. Let $U$ and $V$ be the Jordan domains bounded by $I_1$ and $\Gamma$, and by $I_2$ and $\Theta$, respectively. Because $A$ and $B$, and $B$ and $C$, are adjacent intersection points of $F(S)$ and $I$ by assumption, $\Gamma$ does not cross the interior of $I_2$ and $\Theta$ does not cross the interior of $I_1$. It follows that $U$ and $V$ are either disjoint or nested. On the other hand, since $C$ is minimal and the barrier contains no post-critical points, $U$ and $V$ must each contain at least one post-critical point. Thus $U$ and $V$ are not disjoint; otherwise we obtain a contradiction with the fact that the barrier lake created by $F(S)$ is of type I and therefore does not separate post-critical points. Hence $U$ and $V$ must be nested. Since $[A, B]$ is part of the boundary of $U$, for points near $[A, B]$, only those on one side of $[A, B]$ belong to $U$. To fix ideas, we assume that points close to $[A, B]$ and lying on the right side of $[A, B]$ belong to $U$ (see Figure 7 for an illustration). The following argument can be easily adapted to the other case. Then we must have $U \subset V$. Indeed, in a small neighborhood of $B$, there are points on $\partial V$ and on the left side of $[A, B]$ that do not belong to $\overline{U}$. This implies that there are points in $V$ which are not in $U$. Since $U$ and $V$ must be nested, we conclude $U \subset V$.

  \begin{figure}[!htpb]
  \setlength{\unitlength}{1mm}
  \begin{center}

 \includegraphics[width=100mm]{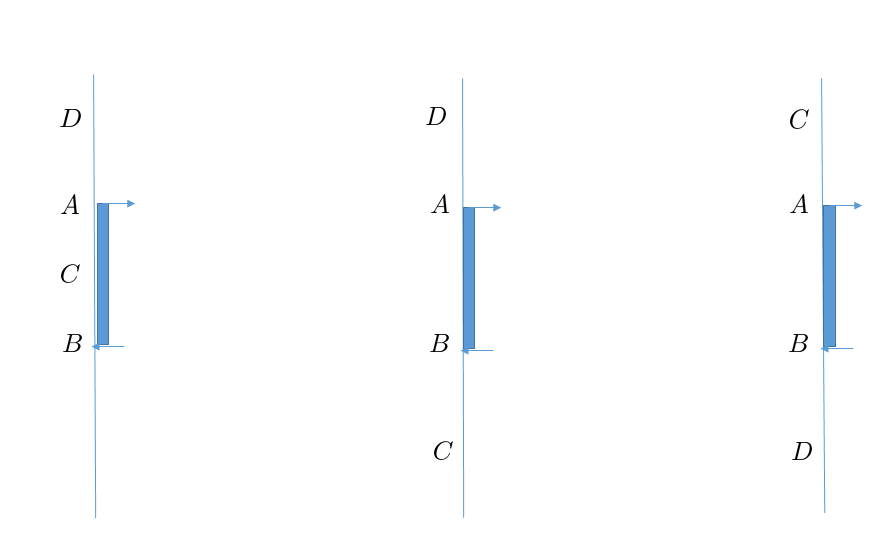}
  \caption{The two adjacent cross directions at $A$ and $B$ are different.  }
  \label{spiral}

  \end{center}
  \end{figure}

The proof of the claim is divided into three cases according to the position of $C$ relative to $A$ and $B$, namely, between $A$ and $B$, below $B$, and above $A$. The argument is the same for all three cases. We provide the details only for the first case and indicate the differences for the other two cases.

 We first suppose that $C$ lies between $A$ and $B$. Since $U \subset V$ and $A \notin \partial V$, it follows that $A \in V$ (see Figures 8 and 9 for an illustration). Let $R$ denote the ray component of $L \setminus \{A\}$ that does not contain $B$. There are two possibilities. The first possibility is illustrated in Figure 8, and the second in Figure 9.

For the first possibility, $R$ first crosses $\Gamma$ at $D$ before it crosses $\Theta$. Then the arc of $\Gamma$ connecting $A$ and $D$, together with the segment of $L$ connecting $A$ and $D$, bound a Jordan domain $X$. A brief reflection shows that $\partial X \cap \partial V = \emptyset$ and hence $\overline{X} \subset V$ and $\overline{X} \cap U = \emptyset$. Since $C$ is minimal, $\overline{X}$ must contain a post-critical point, which then lies in $V$. This contradicts the fact that $U$ and $V$ do not separate post-critical points.

In the second possibility, $R$ crosses $\Theta$ at $D$ before it crosses $\Gamma$. Then the segment of $L$ connecting $D$ and $C$ together with the arc of $\Theta$ connecting $C$ and $D$ bound a Jordan domain $W$. Clearly, $W \cup (A, D) \subset V$ and $(W \cup (A, D)) \cap U = \emptyset$. Since $C$ is minimal, $W \cup (C, D)$ must contain a post-critical point. Because the barrier (and hence $[A, C]$) contains no post-critical point, $W \cup (A, D)$ must contain a post-critical point. This again contradicts the fact that $U$ and $V$ contain the same set of post-critical points.

       In the case where $C$ is below $B$, $A$ is still an interior point of $V$. Again, let $R$ be the ray component of $L \setminus \{A\}$ that does not contain $B$. Then there are two possibilities: either $R$ first crosses $\Gamma$ at $D$ or first crosses $\Theta$ at $D$. For the first possibility, consider the Jordan domain $X$ bounded by $(A, D)$ and the curve segment in $\Gamma$ connecting $A$ and $D$. For the second possibility, consider the Jordan domain $W$ bounded by $(B, D)$ and the curve segment of $\Theta$ connecting $B$ and $D$.

In the case where $C$ is above $A$, let $R$ be the ray component of $L \setminus \{B\}$ that does not contain $A$. Then there are two possibilities: either $R$ first crosses $\Gamma$ at $D$ or first crosses $\Theta$ at $D$. For the first possibility, consider the Jordan domain $X$ bounded by $(B, D)$ and the curve segment in $\Gamma$ connecting $B$ and $D$. For the second possibility, consider the Jordan domain $W$ bounded by $(B, D)$ and the curve segment of $\Theta$ connecting $B$ and $D$.

  \begin{figure}[htbp]
 \centering
 \begin{minipage}[t]{0.40\textwidth}
   \includegraphics[width=\textwidth]{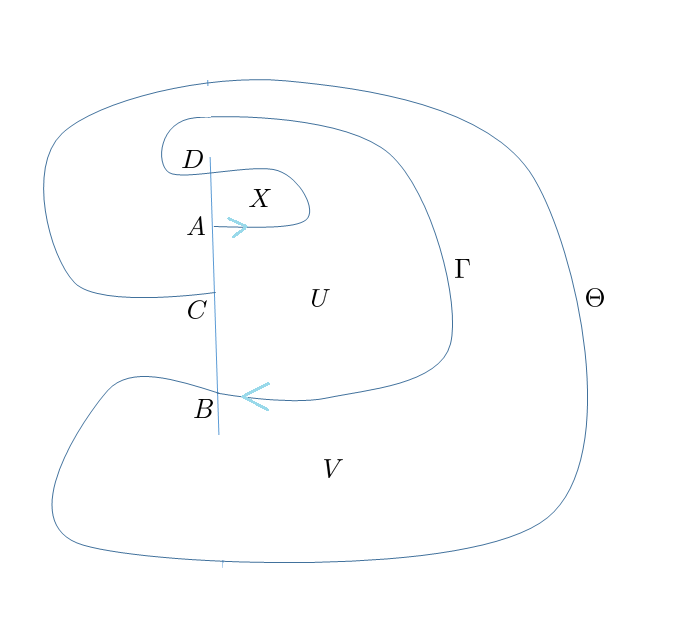}
   \caption{The first possibility for the first case }
   \label{fig:image1}
 \end{minipage}
 \hfill
 \begin{minipage}[t]{0.40\textwidth}
   \includegraphics[width=\textwidth]{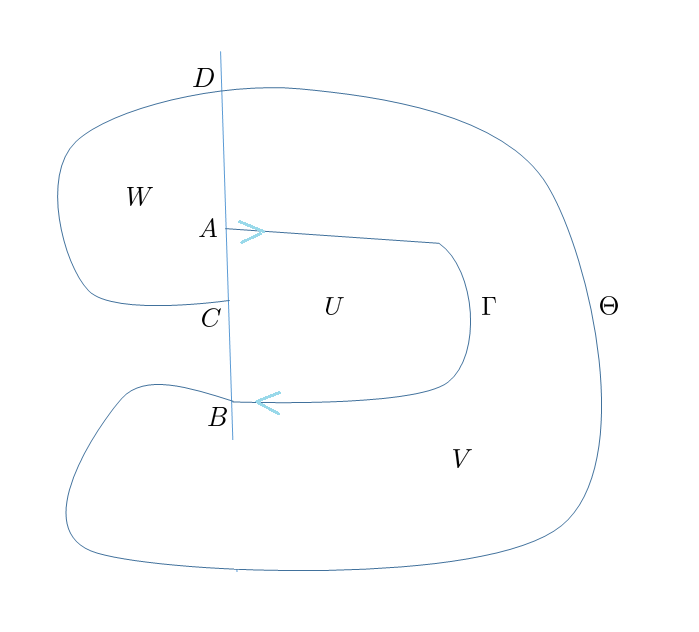}
   \caption{The second possibility for the first case }
   \label{fig:image2}
 \end{minipage}
\end{figure}

The claim has thus been proved. Now we start from the starting point of $F(S)$. By the claim, the cross directions of $F(S)$ and $I$ at the first and second intersection points are the same. Hence the arc of $F(S)$ between the first and second intersection points forms a loop. Next, consider the arc of $F(S)$ between the second and third intersection points of $F(S)$ and $I$. By the claim that the two cross directions are the same and that a type~I barrier lake does not separate post-critical points, we obtain another loop that surrounds the same set of post-critical points as the previous one. Since the two Jordan curves associated with these loops are nested and contain the same set of post-critical points, they must be homotopic to each other in $\mathbb{C} \setminus P_f$. Because $C$ is minimal by assumption, the two loops must have the same intersection number with $L$. The first assertion then follows by continuing this process finitely many times.

      \begin{figure}[!htpb]
  \setlength{\unitlength}{1mm}
  \begin{center}

 \includegraphics[width=55mm]{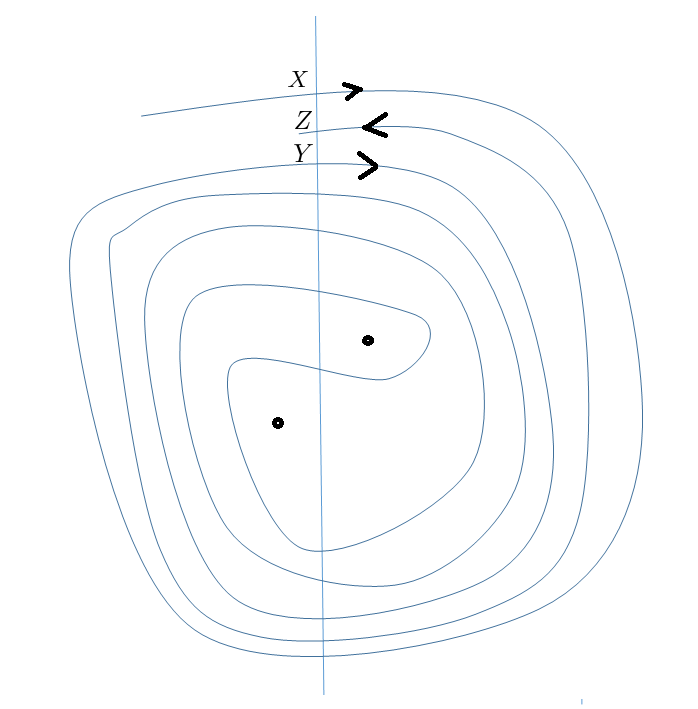}
  \caption{Barrier lake of type I contains at least two post-critical points}
  \label{spiral}

  \end{center}
  \end{figure}

To prove the last assertion, take a piece of $F(S)$, say $\Gamma$, which forms a loop. Suppose the two intersection points of $\Gamma$ with the barrier are $X$ and $Y$, respectively. Since the curve is closed, let $Z$ be the first point along the curve that intersects $(X, Y)$. Then the part of the curve from $X$ to $Z$, together with the segment $[X, Z]$, bounds a Jordan domain $W_1$. Because $C$ is minimal and because the barrier (and thus $[X, Z]$) contains no post-critical point, the Jordan domain must contain a post-critical point $x$. Similarly,  the curve segment from $Y$ to $Z$ and the segment $[Y, Z]$ bounds  a Jordan domain $W_2$, which must also contain a post-critical point, say $y$. Since $W_1$ and $W_2$ are clearly disjoint, we have $x \neq y$. It is clear that $W_1$ and $W_2$ are contained in the sublakes of the barrier lake formed by $F(S)$,  hence $B$ contains at least two post-critical points.

To see the complement of $B$ contains at least one finite post-critical point, we may run along the loop outward and repeat the same argument as above. We get two disjoint Jordan domains in the complement of $B$, each of which is
bounded by a part of $C$ and a part of $L$, with different cross directions. Since $C$ is minimal, the bounded Jordan domain must contain at least one finite post-critical point. The last assertion follows.
\end{proof}

        \section{The thick-thin decomposition}

Recall an $\epsilon$-barrier lake is bounded by the barrier and the quick return $F(S)$.
  We say two barrier lakes are \emph{entangled} with each other if the quick return of  one lake passes through the barrier of the other. The following lemma says, in some sense, that two $\epsilon$-barrier lakes are either disjoint from each other or nested with each other.

\begin{lema}\label{dis}
Let $A$ be an $\epsilon$-barrier lake and let $M(A)$ denote the set of $\epsilon$-barrier lakes $B$ that entangle with $A$. Then
\[
|M(A)| \le 4 C_0 \epsilon N(C, L).
\]
\end{lema}

\begin{proof}
Let $B$ be an $\epsilon$-barrier lake that is entangled with $A$. Then there are two cases.

\textbf{First case:} the quick return part of $\partial B$ passes through the barrier of $A$.
Note that the quick return part of $\partial B$ is determined by a curve segment $S$  in $C^1$ connecting two points, one in $L$ and one in $L'$, which is then determined by the starting point of $S$ lying in $L$. Since $F(S)$ intersects $L$ at most $C_0$ times, for each intersection point of $F(S)$ with the barrier of $A$ there are at most $2C_0$ possibilities for the starting point of $S$. Because the barrier of $A$ contains at most $\epsilon N(C, L)$ intersection points of $C$ and $L$, the number of such $B$ is not more than $2C_0 \epsilon N(C, L)$.

\textbf{Second case:} the quick return part of $\partial A$ passes through the barrier of $B$.
Since the barrier of $B$ contains at most $\epsilon N(C, L)$ intersection points of $C$ and $L$, for each transversal intersection point $z$ of $\partial A$ with $L$, there are at most $2 \epsilon N(C, L)$ possibilities for the starting point of $S$ such that $F(S)$ is the quick return part for $\partial B$ and $z$ belongs to the barrier of $B$. Because the number of transversal intersection points of $\partial A$ with $L$ is at most $C_0$, the number of such $B$ is not more than $2C_0 \epsilon N(C, L)$.

This proves the lemma.
\end{proof}

Let $\mathcal{B} = \{B_i\}$ be a maximal family of type II barrier lakes such that no two lakes in the family are entangled with each other; in other words, for any two lakes in $\mathcal{B}$, either they are disjoint, or one is contained in a sublake of the other. It is possible that $\mathcal{B}= \emptyset$.

\begin{lema}\label{cs}
The number of type II barrier lakes in $\mathcal{B}$ is strictly less than $|P_f|$.
\end{lema}

\begin{proof}
If $\mathcal{B}= \emptyset$, the lemma is trivial. Suppose $\mathcal{B} \ne \emptyset$.
It suffices to construct an injective but not onto map $\Psi: \mathcal{B} \to P_f$. For $B \in \mathcal{B}$, let $\mathcal{F}(B)$ denote the family of elements $B'$ in $\mathcal{B}$ such that either $B' = B$ or $B'$ is contained in some sublake of $B$.

We aim to construct $\Psi$ so that for each $B \in \mathcal{B}$, the following property holds:
\[
\text{For any } B' \in \mathcal{F}(B),\quad \Psi(B') \in B \quad\text{and}\quad |\Psi(\mathcal{F}(B))| < |B \cap P_f|.
\]
Our construction starts from some innermost element of $\mathcal{B}$. Here, ``innermost'' means that it does not contain any other element of $\mathcal{B}$ in its sublake.

To proceed, let $B$ be an arbitrary innermost element in $\mathcal{B}$. Then by Lemma~\ref{obs-3}, $B$ separates two points in $P_f$, say $x$ and $y$. Define $\Psi(B) = x$ and save the point $y$. Clearly, $B$ satisfies the required property. Do the same for all other innermost elements of $\mathcal{B}$. Then remove all innermost elements from $\mathcal{B}$, obtaining a new family $\mathcal{B}'$.

If $\mathcal{B}' = \emptyset$, we are done. Otherwise, repeat the same procedure for $\mathcal{B}'$. Suppose $B'$ is an innermost element of $\mathcal{B}'$. We have two cases.

\emph{Case 1:} $B'$ contains exactly one innermost element $B$ from $\mathcal{B}$. Since $B'$ separates points in $P_f$, $B'$ contains at least one post-critical point $x'$ not contained in $B$. Define $\Psi(B') = x'$. Then the required property holds for $B'$ in this case.

\emph{Case 2:} $B'$ contains $m \ge 2$ innermost elements from $\mathcal{B}$, which together save $m \ge 2$ points in $P_f$. We may define $\Psi(B')$ to be one of these $m$ points, leaving $m-1 \ge 1$ post-critical points saved. Hence the required property also holds in this case.

Proceed similarly for all other innermost elements of $\mathcal{B}'$. After that, remove these innermost elements from $\mathcal{B}'$ to obtain a new family $\mathcal{B}''$. If $\mathcal{B}'' = \emptyset$, we are done. Otherwise, repeat this process again. Finally, we reach an empty family and the theorem follows.
\end{proof}

By Lemma~\ref{obs-2}, except at most
\[
(\epsilon + d^{lp}/C_0) N(C, L)
\]
points, any other intersection point of $C$ and $L$ must belong to some quick return $F(S)$, which by assumption forms either a type I or type II $\epsilon$-barrier lake. By Lemmas~\ref{dis}--\ref{cs}, the number of barrier lakes entangled with elements in $\mathcal{B}$ is at most $4 C_0 \epsilon |P_f| N(C, L)$. Since the number of intersection points of the quick return of each $\epsilon$-barrier lake with $L$ is at most $C_0$, all these together imply:

\begin{lema}\label{er-i}
Except at most
\[
(\epsilon + d^{lp}/C_0 + 4 \epsilon C_0^2 |P_f|) N(C, L)
\]
points, any other intersection point of $C$ and $L$ must belong to some quick return $F(S)$ which produces a type I barrier lake not entangled with any element in $\mathcal{B}$.
\end{lema}

 Now we group the type I barrier lakes guaranteed by Lemma~\ref{er-i} whose loops have the same homotopy type into a class. We label these classes by
\[
A_1, A_2, \dots
\]
such that
\[
N(A_1, L) \ge N(A_2, L) \ge \cdots,
\]
where $N(A_i, L)$ denotes the number of intersection points of all the loops contained in $A_i$ with $L$. Let $n \ge 1$ be the  integer satisfying
\[
N(A_1, L) \ge N(A_2, L) \ge \cdots \ge N(A_n, L) > 4 C_0^2 \epsilon N(C, L) \ge N(A_{n+1}, L).
\]

\begin{lema}\label{obv}
All $A_i$, $1 \le i \le n$, are pairwise disjoint in the sense that for any $1 \le i \ne j \le n$, the barrier lakes in $A_i$ are not entangled with the barrier lakes in $A_j$. In particular, $n \le |P_f| - 3$.
\end{lema}

\begin{proof}
Suppose that for some $1 \le i \ne j \le n$, a type I lake in $A_i$ is entangled with a type I lake $A$ in $A_j$. Since $C$ is minimal and since the loops of the lakes in $A_i$ have the same homotopy type,  all lakes in $A_i$ are entangled with $A$. By Lemma~\ref{dis}, the number of type I barrier lakes contained in $A_i$ is at most $4 C_0 \epsilon N(C, L)$, and hence $N(A_i, L) \le 4 C_0^2 \epsilon N(C, L)$. This contradicts the choice of $n$. The first assertion of the lemma follows.

The second assertion follows from Lemma~\ref{lp1} and the fact that a family of pairwise non-homotopic, disjoint, non-peripheral curves in $\overline{\mathbb{C}} \setminus P_f$ contains at most $|P_f| - 3$ elements.
\end{proof}

     \begin{figure}[!htpb]
  \setlength{\unitlength}{1mm}
  \begin{center}

 \includegraphics[width=150mm]{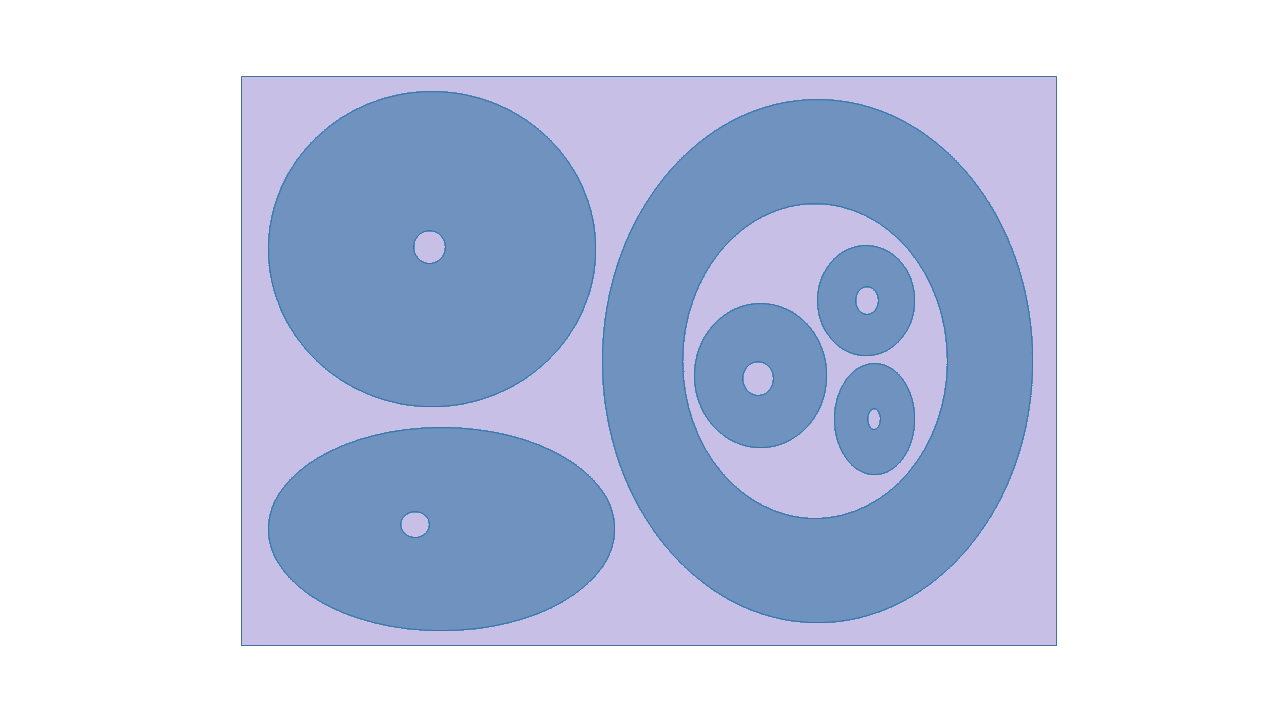}
  \caption{The thick-thin decomposition of the curve $C$. The deep gray part represents the part of $C$ which contributes the most to $N(C, L)$, while the shallow gray part does little contribution to $N(C, L)$.}
  \label{spiral}

  \end{center}
  \end{figure}
Now for each $A_i$ with $1 \le i \le n$, since $C$ is minimal by assumption and all its loops have the same homotopy type, we can take the maximal and minimal loops in $A_i$ so that each piece of $C$ lying between these two loops forms a long spiral consisting of the same number of loops, and $A_i$ is the union of a large number of such long spirals. In other words, if we call the type I barrier lake containing the maximal loop the maximal lake and the one containing the minimal loop the minimal lake, then each time the curve goes into the interior of the maximal lake by passing through its barrier, because $C$ is minimal, it must turn around the minimal lake the same number of times and then enter the interior of the minimal lake through its barrier. The same holds for the reverse order. We still use $A_i$ to denote the union of all these spirals, each of which, up to homotopy, is the union of the same number of adjacent loops with the same homotopy type. We call $A_i$ an annular part of $C$.

By the choice of $n$, for all $i > n$, the type I lakes in $A_i$ must be entangled with some $A_j$ for some $1 \le j \le n$. By Lemmas~\ref{er-i}, \ref{obv} and Lemma~\ref{dis}, it follows that
\[
\sum_{1 \le i \le n} N(A_i, L) \ge \bigl(1 - \epsilon - d^{lp}/C_0 - 8 |P_f| C_0^2 \epsilon\bigr) N(C, L).
\]

\begin{lema6}\label{ann-1} Suppose $C$ is a simple closed curve in $\mathbb{C} \setminus P_f$ and minimal with respect to $\mathcal{F}_0$, and $C^1$ is a component of $F^{-1}(C)$ such that
\[
\min_{\Gamma \sim C^1} N(\Gamma, L) > (1 - \epsilon) N(C, L),
\]
where the minimum is taken over the homotopy class of $C^1$. Then most of $C$ can be organized into finitely many disjoint annuli $A_i$, $1 \le i \le n$ (with $n \le |P_f|-3$), each of which consists of a number of long spirals with the same winding number and homotopy type, such that
\[
N(A_i, L) > 4 \epsilon C_0^2 N(C, L), \qquad 1 \le i \le n,
\]
and
\[
N(C, L) - \sum_{1 \le i \le n} N(A_i, L) \le \bigl(\epsilon + d^{lp}/C_0 + 8 |P_f| C_0^2 \epsilon \bigr) N(C, L),
\]
and  homotopy complexity of any loop in
  these $A_i$, with respect to $\mathcal{F}_0$,    has an upper bound  depending only on  $C_0$.
\end{lema6}

\section{proof of Lemma~\ref{Bridge-Prop}}

Assume that $C$ is minimal such that $\mathcal{N}_{\mathcal{F}_0}(C) > D$ for some large $D$ to be specified at the end. We further assume that there exists a component $C^1$ of $F^{-1}(C)$ such that
\[
\min_{\Gamma \sim C^1} N(\Gamma, L) > (1 - \epsilon) N(C, L).
\]
By the thick--thin decomposition theorem, most of $C$ is organized into disjoint thick annuli, say $A_1, \dots, A_n$. Recall that each $A_i$ is the union of a large number of long spirals with the same homotopy type. By homotopy, we may deform $C$ so that each $A_i$ is contained in an annulus $H_i$ with the same homotopy type and with $P_f \cap H_i = \emptyset$.

Let $t$ be the integer in (\ref{tc}). Now for $m \ge t$ to be specified later, each component of $F^{-m}(H_i)$ is an annulus. We consider only those components whose core curve is non-peripheral. Denote the non-peripheral components of $F^{-m}(H_i)$ by $H_i^1, \dots, H_i^{l_i}$. Then $l_i \le |P_f| - 3$. On the other hand, for each spiral $S$ in $A_i$, each component of $F^{-m}(S)$ is again a spiral. All components of $F^{-m}(S)$ contained in $H_i^j$ form an annular part, which we denote by $A_i^j$. Let $C^m$ be a component of $F^{-m}(C)$. Since $F^m: L \to L$ is a homeomorphism, it follows that
\[
\min_{\Gamma \sim C^m} N(\Gamma, L) \le \sum_{1 \le i \le n,\, 1 \le j \le l_i} N(A_i^j, L) \le \sum_{1 \le i \le n} N(A_i, L).
\]

Let $d_i^j$ denote the degree of $F^m: A_i^j \to A_i$. Note that the length of the loops in $A_i^j$, with respect to the orbifold metric, has a positive lower bound. On the other hand, the complexity of the loops in $A_i$ with respect to $\mathcal{F}_0$ has an upper bound. Therefore, by Lemma~\ref{w-2}, the pullback of the spirals in $A_i$ by $F^t$ yields spirals whose loops have bounded complexity with respect to $\mathcal{F}$, and consequently, up to homotopy in $\Bbb C - P_f$,  bounded length with respect to the orbifold metric. Since $F$ strictly expands the orbifold metric in a neighborhood of the Julia set, and since $F^m = F^{m-t} \circ F^t$, we obtain the following lemma.

\begin{lema}\label{con1}
There exist constants $\lambda, \mu > 0$ depending only on $f$ and $C_0$ such that $d_i^j > \lambda e^{\mu m}$ for all $1 \le i \le n$ and $1 \le j \le l_i$.
\end{lema}

For an annular part $A$ and a long spiral $S$, we denote by $S(A)$ the number of spirals in $A$, by $W(S)$ the winding number of $S$, and by $L(A)$ the number of loops in $A$. Then, up to negligible errors, we have the following.

\begin{lema}\label{winn}
Let $S$ be a spiral in $A_i$ and $S'$ a spiral in $A_i^j$. Then
\[
S(A_i^j) = d_i^j \, S(A_i), \qquad W(S) = d_i^j \, W(S'), \qquad L(A_i) = L(A_i^j).
\]
\end{lema}
Let $C^m$ be a component of $F^{-m}(C)$. The above lemma does not immediately imply a definite reduction from $N(C, L)$ to $N(C^m, L)$. Indeed, although the spirals in $A_i^j$ become shorter, the number of spirals increases while the total number of loops remains unchanged. The key idea of our proof is to show that, except for a bounded number of spirals in each $A_i^j$, all remaining spirals can be removed by deformation in $\mathbb{C} \setminus P_f$. Since Lemma~\ref{con1} guarantees that the number of spirals increases exponentially, by choosing $m$ large enough we ensure that more than half of the spirals in $A_i^j$ become inefficient, hence $N(C^m, L) < \frac{1}{2} N(C, L)$. Lemma~\ref{Bridge-Prop} then follows.

Now let us proceed to the details. To fix ideas, fix $1 \le i \le n$ and $1 \le j \le l_i$. Recall that the annular part $A_i^j$ is contained in the annulus $H_i^j$. Let $U$ denote the unbounded component of $\mathbb{C} \setminus \overline{H_i^j}$. Imagine that an ant travels along a component of $F^{-m}(C)$. Suppose it first travels  a spiral $S'$ in $A_i^j$ from the inside to the outside. To make $S'$ non-homotopically trivial, it must cross the outer boundary component of $H_i^j$ and then turn around at least one post-critical point in $U$ before returning to $H_i^j$. Hence the number of non-trivial spirals in $A_i^j$ is at most twice the number of arcs that turn around one or more post-critical points. Each such arc together with a part of the outer boundary component of $H_i^j$ bounds a Jordan domain in $\mathbb{C}$ containing at least one post-critical point. These Jordan domains are either disjoint or nested. We call the number of Jordan domains in a nest the length of the nest. Clearly, the number of nests is at most $|P_f|$. Let $M'$ denote the maximal length of the nests for $A_i^j$. The notion applies similarly to $A_i$, and we let $M$ be the maximal length of the nests for $A_i$. By deforming $C$ outside all $H_i$ if necessary, we may assume that for all $1 \le i \le n$, whenever an arc of $C$ goes outside $\partial H_i$, it must turn around some post-critical point before returning to $\partial H_i$.

\begin{lema}\label{e-num}
$M' \le M$.
\end{lema}

\begin{proof}
Let $\mathcal{N}$ be a nest of $A_i^j$ of length $M'$. Let $\xi_1, \dots, \xi_{M'}$ be the arcs of $\mathcal{N}$, ordered according to the nest:
\[
\xi_1 \prec \xi_2 \prec \cdots \prec \xi_{M'}.
\]
Let $U_j$ denote the Jordan domain bounded by $\xi_j$ and part of the outer component of $\partial H_i^j$. Then
\[
U_1 \subsetneq U_2 \subsetneq \cdots \subsetneq U_{M'}.
\]
Let $p \in P_f$ be a point belonging to all $U_j$. Then $F^{m}(\xi_1) \subset C$ starts from $\partial H_i$, goes outside $H_i$, may repeatedly go back inside and outside $H_i$, and finally returns to $\partial H_i$ from the outside. By assumption, there is a piece of curve segment of $F^{m}(\xi_1)$ outside $H_i$ that turns around some post-critical point of $f$, say $q$. Take the maximal such piece and denote it by $\eta_1$ (maximal means no other piece of $F^{m}(\xi_1)$ turns around $\eta_1$). Since $\xi_2$ turns around $\xi_1$, there must be some piece of $F^m(\xi_2)$ that turns around $\eta_1$. To see this, take a point $z \in \xi_1$ such that $F^m(z) \in \eta_1$. Since $z$ is an interior point of $U_2$, $F^m(z)$ must be an interior point of a bounded component of $\mathbb{C} \setminus F^m(\partial U_2)$. Because $F^m(\partial U_2)$ is a subset of the union of $F^m(\xi_2)$ and the outer component of $\partial H_i$, there exists an arc in $F^m(\xi_2)$ whose endpoints lie on the outer component of $\partial H_i$, turning around $F^m(z)$ and hence around $\eta_1$. Take the maximal such arc and denote it by $\eta_2$. Continuing this process, we obtain a nest for $A_i$ (not necessarily maximal) such that
\[
\eta_1 \prec \eta_2 \prec \cdots \prec \eta_{M'}.
\]
Thus $M' \le M$, and the lemma follows.
\end{proof}

\begin{figure}[!htpb]
\centering
\includegraphics[width=100mm]{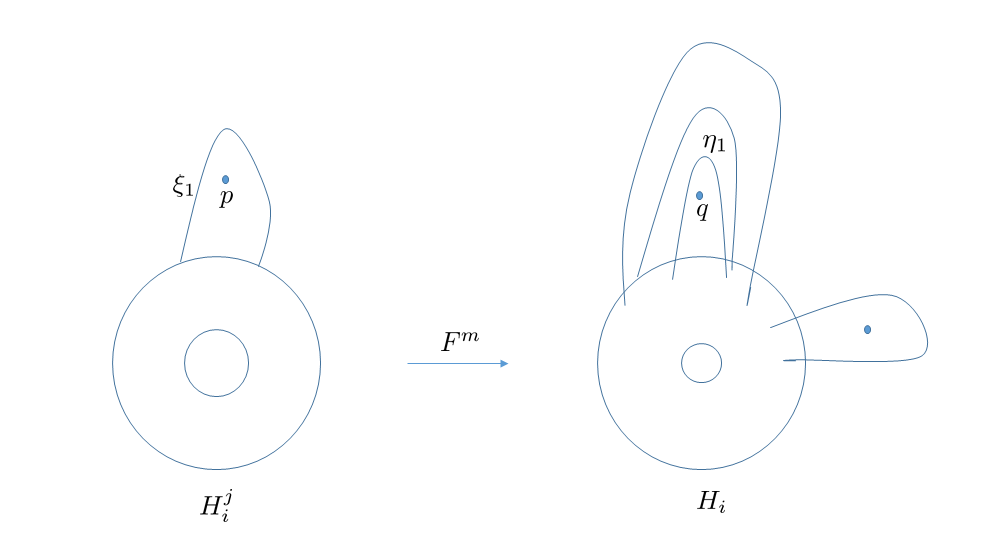}
\caption{$F^{m}(\xi_1)$}
\label{Figure-12}
\end{figure}

Now by Lemma~\ref{con1}  we may   choose $m \ge 1$, depending only on $f$ and $C_0$, such that for all $1 \le i \le n$ and $1 \le j \le l_i$, the degree of
\[
F^m: A_i^j \to A_i
\]
is greater than $2|P_f|$. Since $S(A_i) > 2M$ (each arc in the maximal nest corresponds to two spirals), Lemma~\ref{winn} implies
\[
S(A_i^j) > 4 M |P_f|, \qquad \forall 1 \le j \le l_i.
\]

On the other hand, because the maximal length of the nests for $A_i^j$ is $M'$ and the number of nests is at most $|P_f|$, the number of \emph{efficient} spirals (those that cannot be removed by deformation in $\mathbb{C} \setminus P_f$) is at most $2 M' |P_f| \le 2 M |P_f|$. This implies that at least half of the spirals in each $A_i^j$ ($1 \le i \le n$, $1 \le j \le l_i$) are non-efficient. Consequently, the reduction from $N(C, L)$ to $N(C^m, L)$ is at least half of
\[
\sum_{1 \le i \le n} N(A_i, L),
\]
which accounts for most of $N(C, L)$. Thus $$N(C^m, L) < \frac{1}{2}N(C, L).$$

Since the constant $m$ depends only on $C_0$ and $f$, it remains to specify the constants $\epsilon$, $C_0$, and $D$ before concluding the proof of Lemma~\ref{Bridge-Prop}. To this end, set $\kappa = 1/100$ and define
\[
C_0 = \frac{|P_f| d^{lp}}{\kappa}, \qquad
\epsilon = \frac{\kappa}{8 C_0^2 |P_f|}.
\]
The above argument relies on the fact that the pullback of a spiral in $A_i$ by $F^m$ contains an integer number of loops. Consequently, it suffices to choose $D > 0$ sufficiently large so that the winding number of each spiral in every $A_i^j$ is at least $10$. Lemma~\ref{Bridge-Prop} then follows.

\vspace{0.5cm}
\noindent \textbf{Acknowledgement.} We would like to thank Kevin Pilgrim for introducing us to the finite global attractor problem. We would also like to thank Dan Margalit for his explanation of the complexity of the tree lifting algorithm, which is the motivation for this work.

\end{document}